\documentclass[12pt,a4paper]{article}
\usepackage[T1]{fontenc}
\usepackage[british]{babel}
\usepackage[utf8]{inputenc}
\usepackage{amsmath,amssymb,amsthm}
\usepackage{hyperref}
\newtheorem{thm}{Theorem}[section]
\newtheorem{cor}[thm]{Corollary}
\newtheorem{lem}[thm]{Lemma}
\newtheorem{prop}[thm]{Proposition}

\begin{document}

\begin{center}

{\Large \bf Finite groups with few subgroups not in the Chermak-Delgado lattice}

\end{center}

\vskip0.5cm

\begin{center}

 Jiakuan Lu, Xi Huang, Qinwei Lian

School of Mathematics and Statistics, Guangxi Normal University,\\
Guilin 541006, Guangxi,  P. R. China\\

Wei Meng

School of Mathematics and Computing Science, Guilin University of Electronic
Technology, Guilin, Guangxi, 541002, P.R. China.

\end{center}

\vskip 0.2cm

\begin{abstract}
For a finite group $G$, we denote by $\mathcal{CD}(G)$ the Chermak–Delgado lattice of $G$, and by $v(G)$ the number of conjugacy classes of subgroups of $G$ not in $\mathcal{CD}(G)$.  In this paper, we determine  the finite groups $G$ such that $v(G)=1,2,3$.

\end{abstract}

{\small Keywords: Chermak-Delgado measure; Chermak–Delgado lattice; generalized quaternion 2-group.}

{\small MSC(2000): 20D30, 20D60}


\section{Introduction}

Throughout this paper, $G$ always denotes a finite group and let $\mathcal{L}(G)$ be the subgroup lattice of $G$. For arbitrary subgroups $H$ of $G$, the Chermak-Delgado measure of $H$ is defined as $m_G(H)=|H||C_G(H)|$. Let $m^*(G)={\rm max}\{m_G(H)~|~H\leq G \}$, and $\mathcal{CD}(G)=\{H\leq G~|~m_G(H)=m^*(G) \}$.
Then the set $\mathcal{CD}(G)$ forms a modular self-dual sublattice of $\mathcal{L}(G)$, which is called the Chermak–Delgado lattice of $G$. It was first introduced by Chermak and Delgado \cite{Chermak} and revisited by Isaacs \cite{Isaacs}. In the last years, there has been a growing interest in understanding this lattice. For example, see \cite{An,Brewster1,Brewster2,Brewster3,Brewster4,Fasola2,Glauberman, McCulloch1, McCulloch2, McCulloch3, McCulloch4, Wilcox}.

We remark that the Chermak-Delgado measure associated to a
group $G$ can be seen as a function $H\mapsto m_G(H)$ from $\mathcal{L}(G)$ to $\mathbb{N}^*$. It is clear that
that there is no non-trivial group $G$ such that $\mathcal{CD}(G)=\mathcal{L}(G)$ (see Corollary 3 of \cite{Tarnauceanu1}). In other words, $m_G$ has at least two distinct values for every  non-trivial
group $G$. T\u{a}rn\u{a}uceanu \cite{Tarnauceanu2} determined an interesting class of $p$-groups such that $|{\rm Im}(m_G)|=2$.

The dual problem of finding finite groups with small Chermak–Delgado lattices has
been studied in \cite{McCulloch2, McCulloch3}. We write $\delta(G)$ for the number of subgroups of $G$ not in $\mathcal{CD}(G)$. Hence, $\delta(G)=0$ if and only if $G=1$. Recently, Fasol\u{a} and T\u{a}rn\u{a}uceanu \cite{Fasola} classified all groups where $\delta(G)=1, 2$, and Burrell, Cocke and  McCulloc \cite{Burrell} provided
a classification when $\delta(G)=3, 4$.

\begin{thm} \label{thmy}
Let $G$ be a group. Then

$(1)$ $\delta(G)=1$ if and only if $G\cong Z_p$ where $p$ is a prime or $G\cong Q_8$, the quaternion group of order 8.

$(2)$ $\delta(G)=2$ if and only if $G\cong Z_{p^2}$ where $p$ is a prime.

$(3)$ $\delta(G)=3$ if and only if $G\cong Z_{pq}$ or $G\cong Z_{p^3}$ for primes $p\not=q$.

$(4)$ $\delta(G)=4$ if and only if  $G\cong Z_{p^4}$ where $p$ is a prime, $G\cong Z_{2} \times Z_2$, or $G\cong M_{3^3} =\langle a, b~|~a^{3^{2}}=b^3=1,~b^{-1}ab=a^{4}\rangle$.
\end{thm}
\begin{proof}
Conclusions (1) and (2) are Theorem 1.2 of \cite{Fasola}, (3) and (4) are Theorem C of \cite{Burrell}.
\end{proof}

For a group $G$, we denote by $v(G)$ the number of conjugacy classes of subgroups of $G$ not in $\mathcal{CD}(G)$. Clearly, $v(G)=0$  if and only if $G=1$. In this paper, we are concerned with the groups $G$ such that $v(G)=1, 2, 3$.

\begin{thm} \label{thmm}
Let $G$ be a group. Then

$(1)$ $v(G)=1$ if and only if $G\cong Z_p$ or $Q_8$ where $p$ is a prime.

$(2)$ $v(G)=2$ if and only if $G\cong Z_{q^2}$ where $q$ is a prime, or $$G\cong M_{p^3} =\langle a, b~|~a^{p^{2}}=b^p=1,~b^{-1}ab=a^{p+1}\rangle,$$
where $p>2$ is a prime.

$(3)$ Assume that $G$ is not nilpotent, then $v(G)=3$ if and only if $G$ is a nonabelian group of  $pq$ for primes $p\not=q$.

\end{thm}

\section{Auxiliary results and Lemmas}

In this section, we recall several important
properties of the Chermak–Delgado measure, and prove some auxiliary results. We also collect some Lemmas which are used in the proof of Theorem \ref{thmm}.

\begin{prop}\emph{\cite[Chaper 1G]{Isaacs}\label{prop}} Let $G$ be a group. Then

$(1)$ If $H \le G$, then $m_G(H)\le m_G(C_G(H))$, and if the measures are equal, then
$C_G(C_G(H)) = H$;

$(2)$ If $H \in \mathcal{CD}(G)$, then $C_G(H)\in \mathcal{CD}(G)$ and $C_G(C_G(H)) = H$;


$(3)$ The minimal member $M$ of $\mathcal{CD}(G)$ (called the Chermak–Delgado subgroup of $G$)
is characteristic, abelian and contains $Z(G)$.
\end{prop}

\begin{prop} \label{prop1}
Let $G$ be a group. If $m_G(H)$ divides $|G|$ for arbitrary subgroups $H\leq G$, then $G=1$.
\end{prop}
\begin{proof}
Assume that $G\not=1$, then $G$ has some Sylow $p$-subgroup $P\not=1$ for some prime $p$. It follows that $Z(P)\not=1$, and $m_G(Z(P))=|Z(P)||C_G(Z(P))|$ does not divide $|G|$, a contradiction.
\end{proof}

 By the proof of Proposition \ref{prop1}, one has the following corollary, also see Corollary 3 of \cite{Tarnauceanu1}.

\begin{cor} \label{cor1}
There is no non-trivial group $G$ such that $\mathcal{CD}(G)=\mathcal{L}(G)$.
\end{cor}

\begin{prop} \label{prop2}
Let $G$ be a group. If $|G|$ divides $m_G(H)$ for arbitrary subgroups $H\leq G$, then $G$ is nilpotent.
\end{prop}
\begin{proof}
Let $P$ be a Sylow $p$-subgroup of $G$ for arbitrary prime divisor $p$ of $|G|$. By the assumption, we have that $|G|$ divides  $m_G(P)=|P||C_G(P)|$, and thus, $|G|_{p'}$ divides $|C_G(P)|$, that is, $C_G(P)$ contains some Sylow $q$-subgroup for arbitrary  prime divisor $q\not=p$. By the arbitrariness of $p$, $G$ is nilpotent, as desired.
\end{proof}

\begin{thm} \label{thm1}
Let $G$ be a group. If ${\rm Im}(m_G)$ are consecutive integers, that is, ${\rm Im}(m_G)=\{n,n+1, \cdots, n+k\}$ for some positive integer $n$ and nonnegative integer $k$, then $k=0$, $n=|G|$ and $G=1$.
\end{thm}
\begin{proof}
Note that $m_G(1)=|G| \in {\rm Im}(m_G)$. If  $k\ge 1$, then $|G|+1 \in {\rm Im}(m_G)$ or  $|G|-1 \in {\rm Im}(m_G)$. Thus, there exits a subgroup $H\le G$ such that $m_G(H)=|H||C_G(H)|=|G|+1$ or $|G|-1$. This implies that $|H|$ divides either $|G|+1$ or $|G|-1$. However, since $|H|$ divides $|G|$, these two cases are impossible. Thus, $k=0$ and $n=|G|$.

Assume that $G\not=1$. Then $G$ has some Sylow $p$-subgroup $P\not=1$ for some prime $p$. It follows that $Z(P)\not=1$, and $m_G(Z(P))=|Z(P)||C_G(Z(P))|\not=|G|$, a contradiction.
\end{proof}

The following Theorem is a variation of \cite[Theorem A]{Cocke}, and we use it in the proof of Theorem \ref{thmm}.

\begin{thm} \label{thm2}
Let $G$ be a group. Then the following are equivalent:

$(1)$ $m_G(H)|m_G(K)$ for arbitrary subgroups $H\leq K\leq G$.

$(2)$ $m_G(H)=m_G(H\cap Z(G))$ for arbitrary subgroups $H\leq G$.

$(3)$ $\mathcal{CD}(G)=\{H\leq G~|~Z(G)\leq H \}$. $Q_8\rtimes Z_3$
\end{thm}

\begin{lem}\emph{\cite[Theorem 2.1]{Tarnauceanu2}\label{lem1}}
Let $G$ be a group. For each prime $p$ dividing the order of $G$ and
$P\in {\rm Syl}_p(G)$, let $|Z(P)| = p^{n_p}$. Then
$$|{\rm Im}(m_G)|\ge 1+\sum_p n_p.$$
\end{lem}

Recall that a generalised quaternion 2-group is a group of order $2^n$ for some
positive integer $n \ge 3$, defined by
$$Q_{2^n} =\langle a, b~|~a^{2^{n-1}}=1,~a^{2^{n-2}}=b^2,~b^{-1}ab=a^{-1}\rangle.$$

It is directly checked that $m^*(Q_{2^n})=2^{2n-2}$, and

$$
\mathcal{CD}(Q_{2^n}) =
\begin{cases}
\{Q_8, \langle a \rangle, \langle b \rangle, \langle ab \rangle, \langle a^2 \rangle\} & \text{if } n = 3, \\
\{\langle a \rangle\} & \text{if } n \geq 4.
\end{cases}
$$

\begin{lem}\emph{\cite[Proposition 1.3]{Berkovich}\label{lem2}}
A $p$-group has a unique subgroup of order $p$ if and only if it is
either cyclic or a Generalised Quaternion 2-group.
\end{lem}

The following well known result is a special case of \cite[Theorem 1.2]{Berkovich}.

\begin{lem}\emph{\label{lem3}}
Let $G$ be a nonabelian
$p$-group of order $p^n$  and $p > 2$. Suppose that $G$ has a cyclic subgroup $A=\langle a \rangle$ of index $p$. Then $G$ is isomorphic to $$M_{p^n} =\langle a, b~|~a^{p^{n-1}}=b^p=1,~b^{-1}ab=a^{p^{n-2}+1}\rangle.$$
 \end{lem}
It is directly checked that $m^*(M_{p^n})=p^{2n-2}$, and~$\langle a \rangle \in \mathcal{CD}(M_{p^n})$. Furthermore, $M_{p^n}$ contains precisely one conjugacy class of non-normal subgroups, represented by $\langle b \rangle$.

The following result is a special case of \cite[Theorem 13.7]{Berkovich}.

\begin{lem}\emph{\label{lem4}}
Let $G$ be a $p$-group and $p > 3$. Suppose that $G$
has no normal elementary abelian subgroups of order $p^3$. Then one of the following
holds:

$(1)$ $G$ is metacyclic.


$(2)$ $G=EH$, where $E=\Omega_1(G)$ is nonabelian of order $p^3$ and exponent $p$, $H$ is
cyclic of index $p^2$ in $G$, $Z(G)\le H$ is cyclic, $|G:Z(G)|\le p^3$, $|H:C_G(E)|\le p$.
 \end{lem}

\section{Proof of Theorem \ref{thmm}}

\noindent{\bf Proof of Theorem \ref{thmm}.} (1) If $G\cong Z_p$ or $Q_8$, then clearly $v(G)=1$.

Conversely, assume that $v(G)=1$ and let $H$ be a
representative of the single conjugacy class of subgroups of $G$ which are not in $\mathcal{CD}(G)$. Then $|{\rm Im}(m_G)|=2$. By Lemma \ref{lem1}, $G$ is a $p$-group and $|Z(G)|=p$ for some prime $p$. Set $|G|=p^n$. Then $m_G(1)=p^n$, $m_G(Z(G))=p^{n+1}=m^*(G)$. It follows that $H=1$.  Then $K\in {\cal CD}(G)$ for arbitrary subgroups $K \le G$ with $|K| = p$. This implies that $K=Z(G)$. Thus, $G$ has a unique
subgroup of order $p$. By Lemma \ref{lem2}, $G\cong Z_{p^n}$ or $G\cong Q_{2^n}$. We deduce that $G\cong Z_{p}$ or $G\cong Q_{8}$.\\

(2) If $G\cong Z_{q^2}$ where $q$ is a prime or $G\cong M_{p^3}$ for some prime $p>2$, then clearly $v(G)=2$.

Conversely, assume that $v(G)=2$ and let $H_1, H_2$ be
representatives of the conjugacy class of subgroups of $G$ which are not in $\mathcal{CD}(G)$. Then $|{\rm Im}(m_G)|=2$ or 3.\\

{\bf Case 2.1} Assume that $|{\rm Im}(m_G)|=2$. Again by Lemma \ref{lem1}, $G$ is a $p$-group and $|Z(G)|=p$ for some prime $p$. Set $|G|=p^n$. Then $m_G(1)=p^n$, $m_G(Z(G))=p^{n+1}=m^*(G)$. Thus, one of
$H_1$ and $H_2$ must be trivial, say $H_1 = 1$. If $|H_2|>p$, then all subgroups of order $p$ are in ${\cal CD}(G)$. It follows that $Z(G)$ is the unique subgroup of order $p$. By Lemma \ref{lem2}, $G\cong Z_{p^n}$ or $G\cong Q_{2^n}$, none of which could satisfy the hypothesis $v(G)=2$. Assume that $|H_2|=p$. Clearly $H_2\not=Z(G)$. Then $m_G(H_2)=p|C_G(H_2)|=p^n$, and so the length of the conjugacy class containing $H$ is equal to $p$. It follows that $G$ has exactly $p+1$ subgroups of order $p$. Thus, $G$
has no normal elementary abelian subgroups of order $p^3$. Assume that $p>3$. Clearly, the groups in (2) of Lemma \ref{lem4} could not satisfy the hypothesis. If $G$ is metacyclic, then $G$ has a cyclic subgroup of index $p$
since $Z(G)$ contained in every number of $\mathcal{CD}(G)$. By Lemm \ref{lem3},
 $G\cong M_{p^n}$. It follows that $m^*(G)=m_G(\langle a \rangle)=p^{2n-2}$. So $n=3$ and $G\cong M_{p^3}$. Assume that $p=3$, or $2$. By Theorem \ref{thmy}, $G\cong M_{3^3}$.\\

{\bf Case 2.2} Assume that $|{\rm Im}(m_G)|=3$. Then $m_G(H_1)$, $m_G(H_2)$ and $m^*(G)$ are distinct.
By Lemma \ref{lem1}, we discuss two cases.\\

{\bf Case 2.2.1} $|G| = p^n$ and $G$ has the centre of
orders $p$ or $p^2$.

Clearly, if $G$ is abelian, then $G\cong Z_{p^2}$. Assume that $G$ is not abelian. Then $m_G(1)<m_G(Z(G))=m_G(G)=m^*(G)$, and one of
$H_1$ and $H_2$ must be trivial, Say $H_1=1$.

Assume that $|Z(G)|=p$. If $|H_2|>p$, and all subgroups of order $p$ are in ${\cal CD}(G)$. It follows that $Z(G)$ is the unique subgroup of order $p$ and so $G\cong Q_{2^n}$, which could not satisfy $v(G)=2$.
Thus, $|H_2|=p$. By Theorem \ref{thm2}, $m_G(1)$ divides $m_G(H_2)$, so $m_G(H_2)=p^n$. This contradicts that $|{\rm Im}(m_G)|=3$.

Assume that $|Z(G)|=p^2$. Then all subgroups of order $p$ are outside of ${\cal CD}(G)$. By the hypothesis, $G$ has a unique subgroup of order $p$. The same can be said as above.\\

{\bf Case 2.2.2}. $|G| = p^nq^m$, $|Z(P)|=p$ and $|Z(Q)|=q$, where $P$ and $Q$ are the Sylow $p$-subgroups and Sylow $q$-subgroups of $G$, respectively.

We have that
 $$m_G(Z(P)) = p |C_G(Z(P))| = p^{n+1}q^x,$$ where $0 \le x \le m$,
and similarly, $$m_G(Z(Q)) =q|C_G(Z(Q))|=p^yq^{m+1},$$ where $0 \le y \le n$. Also,
 $$m_G(1) = p^nq^m,~m_G(G)=p^nq^m|Z(G)|.$$
Clearly, ${\rm Im}(m_G)=\{m_G(Z(P)), m_G(Z(Q)), m_G(1)\}$.

If $m^*(G)=m_G(1)=p^nq^m$, then $Z(P)$, $P$, $Z(Q)$, $Q$  are not in ${\cal CD}(G)$. By the hypothesis, $Z(P)=P$, and $Z(Q)=Q$, that is, $G$ has order $pq$. Clearly, in this case, $G \in {\cal CD}(G)$, and $Z(G)=1$. It is no loss to assume that $p>q$. Then $m_G(P)=p^2>m^*(G)$, a contradiction. Thus, $m^*(G)=m_G(Z(P))$ or $m_G(Z(Q))$.

In what follows, without loss of generality, assume that $m^*(G)=m_G(Z(P))$. Then $Z(P)\in {\cal CD}(G)$ and $Z(G)\le Z(P)$ by (3) of Proposition \ref{prop}. Clearly, $1$, and $Z(Q)=Q$ are not in ${\cal CD}(G)$.
 If $Z(G)=1$, then $G$ will not in ${\cal CD}(G)$, which contradicts that $v(G)=2$. If $Z(G)=Z(P)$ has order $p$,
then $$m^*(G)=m_G(Z(G))=m_G(Z(P))=m_G(P)=m_G(G)=p^{n+1}q.$$
This implies that $C_G(P)$ contains some Sylow $q$-subgroup, and so $G=P\times Q$, a contradiction.\\



(3)  If $G$ is a nonabelian group of  $pq$ for primes $p\not=q$, then $v(G)=3$.

Conversely, assume that $G$ is not nilpotent and $v(G)=3$. Let $H_1, H_2$ and $H_3$ be
representatives of the conjugacy class of subgroups of $G$ not in $\mathcal{CD}(G)$. Then $|{\rm Im}(m_G)|=3$ or 4.\\

{\bf Case 3.1}. Assume that $|{\rm Im}(m_G)|=3$. Then $|G| = p^nq^m$, $|Z(P)|=p$ and $|Z(Q)|=q$, where $P$ and $Q$ are the Sylow $p$-subgroups and Sylow $q$-subgroups of $G$, respectively.

 We have that  ${\rm Im}(m_G)=\{m_G(Z(P)), m_G(Z(Q)), m_G(1)\}$, as in Case 2.2.2.

 If $m^*(G)=m_G(1)=p^nq^m$, then $Z(P)$, $P$, $Z(Q)$, $Q$ are not in ${\cal CD}(G)$. Since $v(G)=3$, we have that either $Z(P)=P$, or $Z(Q)=Q$. Say $Z(P)=P$. Assume that $Z(Q)<Q$. Since $|Z(Q)|=q$, there exists a subgroup $K$ such that $Z(Q)<K<Q$. Then $K\in{\cal CD}(G)$ and so $m_G(K)=|K||C_G(K)|=pq^m$. This implies that $C_G(K)$ contains some Sylow $p$-subgroup, say $P\le C_G(K)$. Thus, $Z(Q)<K\le C_G(P)$ and $Z(Q)\le Z(G)$. It follows  that $m_G(Z(G))=|Z(G)|pq^m>m^*(G)$, a contradiction. Thus, $Z(Q)=Q$, that is, $G$ has order $pq$. It is no loss to assume that $p>q$. Then $m_G(P)=p^2>m^*(G)$, a contradiction.

 In what follows, without loss of generality, assume that $m^*(G)=m_G(Z(P))$. Then $Z(P)\in {\cal CD}(G)$ and $Z(G)\le Z(P)$ by (3) of Proposition \ref{prop}. If $Z(G)=1$, then $1$, $Z(Q)$, and $G$ will not in ${\cal CD}(G)$.  Since $v(G)=3$, we have that $Z(Q)=Q$, and $m^*(G)=m_G(Z(P))=p^{n+1}$.
Clearly, $P\in {\cal CD}(G)$ and so $P\trianglelefteq G$. If $Z(P)<P$, then $Z(P)Q\in {\cal CD}(G)$. However, $q$ divides $m_G(Z(P)Q)$, and $m_G(Z(P)Q)\not=m^*(G)$, a contradiction. Thus, $Z(P)=P$, that is, $G$ has order $pq$.

 If $Z(G)\not=1$, then $m_G(1)<m_G(G)$, and so $m_G(G)=m_G(Z(P))$ or $m_G(Z(Q))$. It is no loss to assume that $m_G(G)=m_G(Z(P))$. Then $|Z(G)|=p$. It follows that $Z(G)=Z(P)$. By (3) of Proposition \ref{prop}, $m^*(G)\not=m_G(Z(Q))$. It follows that $1$ and $Z(Q)$ are not in ${\cal CD}(G)$, and $$m^*(G)=m_G(Z(G))=m_G(Z(P))=m_G(G)=p^{n+1}q^m.$$
 Since $p^{n+1}$ divides $m_G(P)$, we have that $m_G(P)=m^*(G)$ and $P\in{\cal CD}(G)$. This implies that $Q\le C_G(P)$, and so $G=P\times Q$, a contradiction.\\

{\bf Case 3.2}. Assume that $|{\rm Im}(m_G)|=4$. Then $m_G(H_1)$, $m_G(H_2)$ $m_G(H_3)$ and $m^*(G)$ are distinct.
By Lemma \ref{lem1}, we discuss two cases.\\

 {\bf Case 3.2.1}. $|G| = p^nq^m$, $|Z(P)|=p$ or $p^2$ and $|Z(Q)|=q$, where $P$ and $Q$ are the Sylow $p$-subgroups and Sylow $q$-subgroups, respectively.\\

If $|Z(P)|=p$, we we derive a contradiction similar to Case 3.1. In what follows, we assume that $|Z(P)|=p^2$ and let $P_1$ be a subgroup of  $Z(P)$ of order $p$.

If $m^*(G)=m_G(1)=p^nq^m$, then $Z(P)$, $P$, $Z(Q)$, $Q$ are not in ${\cal CD}(G)$. Since $v(G)=3$, we have that $Z(P)=P$, $Z(Q)=Q$, and $G$ has order $p^2q$. Thus, $Q\le GL(2,p)$ and so $q=p+1$ or $q\le p-1$. This implies that $m_G(P)=p^4>p^2q=m^*(G)$, a contradiction.

If $m^*(G)=m_G(Z(Q))$, then 1, $P_1$ and $Z(P)=P$ are not in ${\cal CD}(G)$. If $Z(G)=1$, then $G$ will not in ${\cal CD}(G)$, contradicting that $v(G)=3$. Thus, $Z(G)=Z(Q)$ by (3) of Proposition \ref{prop}, and so $m^*(G)=m_G(Z(Q))=p^2q^{m+1}$.
Clearly, $Z(Q)P_1\in {\cal CD}(G)$. However, $p^3$ divides $m_G(Z(Q)P_1)$, and so $m_G(Z(Q)P_1)\not=m^*(G)$, a contradiction.

 If $m^*(G)=m_G(Z(P))$, then 1, $P_1$ and $Z(Q)=Q$ are not in ${\cal CD}(G)$. Clearly, $G\in {\cal CD}(G)$.
 So $Z(G)=Z(P)$ by (3) of Proposition \ref{prop}, and so $m^*(G)=m_G(Z(P))=p^{n+2}q$.
Clearly, $QP_1\in {\cal CD}(G)$. However, $m_G(Z(Q)P_1)\not=m^*(G)$, a contradiction.

If $m^*(G)=m_G(P_1)$, then 1, $Z(P)=P$ and $Z(Q)=Q$ are not in ${\cal CD}(G)$. Clearly, $G\in {\cal CD}(G)$.
 So $Z(G)=P_1$ by (3) of Proposition \ref{prop}, and so $m^*(G)=m_G(P_1)=p^{3}q$.
Clearly, $QP_1\in {\cal CD}(G)$. However, $m_G(QP_1)\not=m^*(G)$, a contradiction.\\


 {\bf Case 3.2.2}. $|G| = p^nq^mr^t$, $|Z(P)|=p$, $|Z(Q)|=q$, $|Z(R)|=r$, where $P$ $Q$ and $R$ are the Sylow $p$-subgroups, Sylow $q$-subgroups and Sylow $r$-subgroups of $G$, respectively.

We have that
$${\rm Im}(m_G)=\{m_G(Z(P)), m_G(Z(Q)), m_G(Z(R)), m_G(1)\}.$$

If $m^*(G)=m_G(1)$, then $Z(P)=P$, $Z(Q)=Q$ and $Z(R)=R$ are not in ${\cal CD}(G)$. Clearly, $G\in {\cal CD}(G)$ and $Z(G)=1$, and $|G| = pqr$. Since $G$ is solvable, $G$ has a normal maximal subgroup, say $QR$. Then $QR\in {\cal CD}(G)$. This implies that $P\le C_G(QR)$ and $QR\le C_G(P)$. It follows that $P\le Z(G)$, a contradiction.

 In what follows, it is no loss to assume that $m^*(G)=m_G(Z(P))$. Then $Z(P)\in {\cal CD}(G)$ and $Z(G)\le Z(P)$ by (3) of Proposition \ref{prop}. Clearly,  $1$, $Z(Q)=Q$ and $Z(R)=R$ are not in ${\cal CD}(G)$.
 If $Z(G)=1$, then $G$ will not in ${\cal CD}(G)$, which contradicts that $v(G)=3$. Thus, we have that $Z(G)=Z(P)$, and $m^*(G)=m_G(Z(P))=m_G(Z(G))=m_G(G)=p^{n+1}qr$.

Clearly, $P\in {\cal CD}(G)$ and so $P\trianglelefteq G$. Thus, $P$ has a $p$-complement in $G$, say $QR$. Then $QR\in {\cal CD}(G)$. However, $p^{n+1}$ does not divide $m_G(QR)$, and so $m_G(Z(P)Q)\not=m^*(G)$, a contradiction.

The proof is completed. $\Box$\\

\noindent{\bf Availability of data and materials}\\

The datasets supporting the conclusions of this article are included within the article.\\

\noindent{\bf Acknowledgements}\\

The first author is supported by the Guangxi Natural Science Foundation Program (2024GXNSFAA010514).

\bibliographystyle{amsplain}

\end{document}